\documentclass[final,onefignum,onetabnum]{article}

\usepackage{amsmath}
\usepackage{amssymb}
\usepackage{graphicx}
\usepackage{color}
\usepackage{subfig}
\usepackage{bm}
\usepackage{tabularx}

\usepackage{ntheorem,cleveref}
 \newtheorem{theorem}{Theorem}[section]
\newtheorem{lemma}[theorem]{Lemma}

\newenvironment{proof}{\paragraph{Proof:}}{\hfill$\square$}

 \numberwithin{equation}{section}

\DeclareMathOperator*\dif{\mathop{}\!\mathrm{d}}
\makeatletter

\newcommand{\Rmnum}[1]{\expandafter\@slowromancap\romannumeral #1@}
\makeatother

\newcommand{\T}{~0\leq t\leq T}
\newcommand{\R}{\mathbb{R}}
\newcommand{\ds}{\,{\dif}s}

\newcommand{\dr}{\,{\dif}r}
\newcommand{\dd}{\,{\dif}}
\newcommand{\dwr}{\,{\dif}W_r}
\newcommand{\dws}{\,{\dif}W_s}

\newcommand{\dbs}{\,{\dif}\overleftarrow{B_s}}

\newcommand{\dbr}{\,{\dif}\overleftarrow{B_r}}
\newcommand{\dbi}{\Delta B_i}
\newcommand{\dwi}{\Delta W_i}

\newcommand{\tp}{t_{i+1}}

\newcommand{\ii}{\int_{t_i}^{t_{i+1}}}
\newcommand{\iit}{\int_{t}^{t_{i+1}}}
\newcommand{\iis}{\int_{s}^{t_{i+1}}}

\newcommand{\ya}{\tilde{Y}_i}
\newcommand{\yb}{Y_i}
\newcommand{\z}{\tilde{Z}_i}

\newcommand{\yatp}{Y_{i+1}(t_{i+1})}
\newcommand{\fs}{f(s,X_s,Y_s,Z_s)}
\newcommand{\fr}{f(r,X_r,Y_r,Z_r)}
\newcommand{\ffs}{f(s,X_s,\ya(s),\z(s))}

\newcommand{\ffr}{f(r,X_r,\ya(r),\z(r))}
\newcommand{\gs}{g(s,X_s,Y_s)}
\newcommand{\gr}{g(r,X_r,Y_r)}
\newcommand{\ggs}{g(s,X_s,\yb(s))}

\newcommand{\E}{E_{t_i}}

\newcommand{\f}{\frac}

\newcommand{\ola}{\overleftarrow}

\begin{document}

\title{Solving Backward Doubly Stochastic Differential Equations through Splitting Schemes}
\author{Feng Bao \thanks{ Department of Mathematics,  Florida State University, 1017 Academic Way, Tallahassee, FL 32304 ({\tt fbao@fsu.edu}). }
        \and Yanzhao Cao \thanks{Department of Mathematics, Auburn University, Auburn, AL 36849
         \ ({\tt yzc0009@auburn.edu}).}
         \and   He Zhang \thanks{Department of Mathematics, Auburn University, Auburn, AL 36849
         \ ({\tt hzz0077@auburn.edu}).} }
         \date{}
\maketitle

\begin{abstract}
A splitting scheme for  backward doubly stochastic differential equations  is proposed. The main idea is to decompose  a    backward doubly stochastic differential equation  into a   backward  stochastic differential equation and a stochastic differential equation. The backward stochastic differential equation  and  the  stochastic differential equation are then approximated by first order finite difference schemes, which results in a first order scheme for the  backward doubly stochastic differential equation.  Numerical experiments are conducted  to illustrate the convergence rate of the proposed scheme.
\end{abstract}



\section{Introduction}
The aim of this paper is to introduce a splitting algorithm for the following backward doubly stochastic differential equation (BDSDE):
\begin{equation}\label{1b}
    Y_t  = \  \xi+\int_t^T \fs\ds-\int_t^T Z_s\dws +\int_t^T \gs\dbs,
\end{equation}
where $0\leq t\leq T$, $W:=\{W_t\}_{t\geq0}$, $B:=\{B_t\}_{t\geq0}$ are two independent Brownian motions and the stochastic process $X_t$ is defined by $X_t=X_0+W_t$, where $X_0$ is an initial random variable independent of $W$ and $B$.  The notation  ${\dif}\overleftarrow{B}$  stands for the backward It\^{o} integral (see \cite{Pardoux1987}), which is an It\^o integral with backward propagation direction.  The solution of the BDSDE \eqref{1b} is a pair of stochastic processes $(Y_t, Z_t)$. Here ``doubly" refers to the fact that the equation is driven by two independent Brownian motions. Without the $d\ola{B}_t$ integral, the BDSDE is reduced to a standard backward stochastic differential equation (BSDE), which has been extensively studied \cite{ma2002numberical, ma1994solving, pardoux1992backward, zhang2004, zhang2017backward}.

The theory of BDSDEs was  first studied  in \cite{pardoux1994backward} to give a probabilistic interpretation
for the solutions of the following class of semilinear stochastic partial differential equations (SPDEs)
\begin{equation}\label{spde}
   \begin{split}
        u(t,x)= &\Phi(x)+\int_t^T \big(\mathcal{L}u(s,x)+f(s,x,u(s,x),(\nabla u\sigma)(s,x))\big)\dd s \\
       & \quad +\int_t^T g(s,x,u(s,x))\dbs,~~~(t,x)\in [0,T]\times\R^d
   \end{split}
\end{equation}
through  the relation
 \begin{equation}\label{equi}
   Y_t=u(t,X_t),~~Z_t=\nabla u(t,X_t)\sigma(X_t).
 \end{equation}
The SPDE system \eqref{spde} provides a stochastic version of parabolic type PDEs which could decribe uncertainties  in modeling physical and engineering problems.  For example, in the case that $f$ is a linear function, the above SPDE solves the optimal filstering problem which aims to obtain the best estimate for the state of some stochastic dynamical system based on noisy partial observational data \cite{bao2014_JUQ}. The optimal filtering problem is the key mission in data assimilation and it has been widely used in target tracking, weather forecasting, image processing, parameter estimation, etc.. In  an optimal filtering problem, we need to obtain the conditional expectation for the target dynamical system given the observational information. It was proved (\cite{zakai1969optimal}) that the solution of the SPDE system \eqref{spde} (in the linear case) is the conditional probability density for the dynamical system in the optimal filtering problem, which is used to calculate the desired conditional expectation. In the connection of the equivalence relation \eqref{equi}, the BDSDE \eqref{1b}  also provides solution for the optimal filtering problem. In a recent study (\cite{bao2020_Jump, bao2019_Jump, bao2019, bao2011_BDSDE}), we established a direct link between BDSDEs and optimal filtering problems. The main advantage of solving application problems via BDSDEs instead of SPDEs is twofold. First, solving BDSDEs is mesh free, thus unstructured methods such as Monte Carlo methods and stochastic meshfree approximations can be applied \cite{Bao_2017}. Moreover, scalable parallel numerical algorithms for BSDEs and BDSDEs enable us to benefit from recent advances in high performance parallel computing and even the deep learning techniques (\cite{E2017, doi:10.1137/16M106371X,labart2013parallel}).  Second, while it is very difficult to construct higher order methods to solve SPDEs, high order schemes for BDSDEs are relatively easy to construct (\cite{ bao2015_First, bao2018_First, bao2016first}).

In this paper we introduce  a numerical scheme for the  BDSDE \cref{1b} using the splitting up method.
Our work is inspired by the studies of splitting up method for linear SPDEs. The application of the splitting up methods to linear SPDEs  was initiated by A. Bensoussan et al \cite{bensoussan1989approximation} where the SPDE is decomposed into a PDE and an SDE. Bensoussan's method was  further developed in \cite{bensoussan1990approximation,bensoussan1992approximation,legland1992splitting}. In particular, Gy{\"o}ngy and Krylov \cite{gyongy2003splitting}, proved  the convergence in the maximum norm.

To obtain a splitting up approximation for the BDSDE \cref{1b}, we decompose it into two equations, a BSDE which serves as a predictor or a  pre-solving procedure,  and an SDE which serves as an update procedure.  Both can be solved using highly efficient numerical schemes (\cite{GOBET20171171,kloeden2013numerical,JCM-31-221,zhang2017backward}). In this paper, we construct a first order scheme by using the Milstein scheme on the SDE and a simple first order scheme on the BSDE.  One of the advantages of our splitting up schemes,  in comparison with the existing numerical schemes for BDSDEs (\cite{Bachouch2016, bao2016first}), is that it avoids the solve of  $Z_t$ in \cref{1b}, which significantly reduces the computing cost. It's also worthy to point out that the conventional splitting up methods under the SPDEs framework are focused on the case that both $f$ and $g$ in \cref{spde} are linear functions while our methodology applies to more general nonlinear equations.
In addition, the significance our splitting up method is boosted by some recent work  of E, Han and Jentzen (\cite{E2017,Han2018}), where a deep learning technique is used to solve fairly high dimensional  BSDEs.  Such a method can be applied to solve the BSDE, which is the most computational expensive component in our  splitting up algorithm, thus can help solve high dimensional BDSDEs through our splitting up process.

The rest of this paper is organized as follows. In Section 2, we introduce some notations, assumptions  and concepts as well as some known theoretical results of BDSDEs.   In Section 3, we first present the splitting up method where the BDSDE is split into a BSDE and an SDE, and then prove the first order convergence. The numerical schemes with the corresponding numerical analysis are presented in Section 4, followed by three numerical examples in Section 5.

\section{Preliminaries}
Let $T>0$ be a fixed terminal time,  $(\Omega,\mathcal{F},P)$  a probability space, and   $W$ and $B$  two mutually independent Brownian motions on this space, with values in $\R^d$ and $\R^l$, respectively. For each $t\in[0,T]$, define two collections $\{\mathcal{F}_t\}_{0\leq t\leq T}$ and $\{\mathcal{G}_t\}_{0\leq t\leq T}$  by
\begin{equation*}
  \mathcal{F}_t:=\mathcal{F}_{0,t}^W\vee \mathcal{F}_{t,T}^B,~\text{and}~\mathcal{G}_{t}:=\mathcal{F}_{0,t}^W\vee \mathcal{F}_{0,T}^B,
\end{equation*}
where $\mathcal{F}_{s,t}^W$ and $\mathcal{F}_{s,t}^B$ are the
completion of $\sigma\{W_r-W_s;s\leq r\leq t\}$ and $\sigma\{B_r-B_s;s\leq r\leq t\}$ , respectively. Here $\{\mathcal{F}_t\}_{0\leq t\leq T}$ is neither increasing nor decreasing, while $\{\mathcal{G}_t\}_{0\leq t\leq T}$ is an increasing filtration. To simplify the presentation and make our analysis more readable, we assume throughout  the paper that $d=l=1$. The results obtained in this paper can be extended to multi-dimensional cases through similar procedures.

Denote by $\mathcal{M}^2([0,T];\R)$ the set of all $\R$-valued, $\mathcal{F}_t$-measurable processes $\{\varphi(t)\}_{0\leq t\leq T}$ such that $E\int_0^T |\varphi(t)|^2\,{\dif}t<\infty$, by $\mathcal{S}^2([0,T];\R)$ the set of all $\mathbb{R}$-valued, $\mathcal{F}_t$-measurable processes $\{\varphi(t)\}_{0\leq t\leq T}$ such that $E\Big[\sup\limits_{0\leq t\leq T} |\varphi(t)|^2\Big]<\infty$, and by $L^2(\Omega,\mathcal{F}_T,P;\R)$ the set of all $\mathcal{F}_T$-measurable random variable $\xi$ such that $E|\xi|^2<\infty$.

We assume that  $\Phi$, $f$ and $g$ satisfy the following regularity assumptions:


\textbf{(H1)}~$\Phi\in C^3(\R,\R)$, $f\in C^3([0,T]\times\R\times\R\times\R,\R)$, and $g\in C^3([0,T]\times\R\times\R,\R)$. Here $C^k(A,B)$ denotes the set of functions of class $C^k$ from $A$ to $B$ whose partial derivatives of order less than or equal to $k$ are bounded.

\textbf{(H2)}~$f\colon\Omega\times[0,T]\times\R\times\R\times\R\to\R$ and $g\colon\Omega\times[0,T]\times\R\times\R\rightarrow\R$ are jointly measurable. For any $x, y, z \in\R$,
\begin{equation}
  f(\cdot,x,y,z)\in\mathcal{M}^2([0,T];\R),~\text{and}~g(\cdot,x,y)\in\mathcal{M}^2([0,T];\R).
\end{equation}

\textbf{(H3)}\label{h3}~$f$ and $g$ satisfy the Lipschitz conditions. For all $\omega\in \Omega$, $t,s\in[0,T]$,~$x,\bar{x}\in\mathbb{R}$,~$y,\bar{y}\in\mathbb{R}$, $z,\bar{z}\in\mathbb{R}$, there exists a constant $L>0$ such that
\begin{equation}\label{33}
  \begin{split}
     & |f(t,x,y,z)-f(s,\bar{x},\bar{y},\bar{z})|^2\leq L(|t-s|+|x-\bar{x}|^2+|y-\bar{y}|^2+|z-\bar{z}|^2), \\
      & |g(t,x,y)-g(s,\bar{x},\bar{y})|^2\leq L(|t-s|+|x-\bar{x}|^2+|y-\bar{y}|^2).\\
  \end{split}
\end{equation}
 Moreover,
\begin{equation}\label{42}
  \sup\limits_{\T}\{|f(t,0,0,0)|^2+|g(t,0,0)|^2\}<L.
\end{equation}

The following theorem is a collection of well posedness and regularity results on BDSDEs which will be used throughout the rest of the paper.

\begin{theorem}\label{th1}
Let \textbf{(H1)-(H3)} hold.

(1) (Theorem 1.1 in \cite{pardoux1994backward}) For any $\Phi(X_T)\in L^2(\Omega,\mathcal{F}_T,P;\R)$, BDSDE \eqref{1b}
has a unique solution $(Y,Z)\in\mathcal{S}^2([0,T];\R)\times\mathcal{M}^2([0,T];\R)$.

(2) (Theorem 1.4 in \cite{pardoux1994backward}) There exists a positive constant $M$, independent of t, such that
\begin{equation*}
E\bigg[\sup\limits_{\T}|Y_t|^2 + \int_{0}^{T} |Z_t|^2 \dd t\bigg]\leq M.
\end{equation*}

(2) (Lemma 4.2 in \cite{bao2016first})~For $0\leq s\leq t\leq T$, there exists some positive constant $C$, independent of t, such that
\begin{equation*}
  \begin{split}
  &E[(Y_t-Y_s)^2] \leq C(t-s), \hspace{1em} | E[Y_t-Y_s] | \leq C(t-s).
  \end{split}
\end{equation*}

(3) (Lemma 2.3 in \cite{pardoux1994backward})~For any $t\leq s\leq T$, $(\nabla Y_s, \nabla Z_s)$ is the unique solution of the following variational equation
\begin{equation}\label{z}
\begin{split}
  &\nabla Y_s=\Phi'(X_T)\nabla X_T+\int_s^T\nabla \fr\dr- \int_s^T\nabla Z_s\dwr+\int_s^T \nabla \gr\dbr,
  \end{split}
\end{equation}
where $\nabla$ is the gradient operator with respect to $X_0$ ($X_0$ denoting the initial condition for $X_t$),
\begin{equation*}
  \begin{split}
  &\nabla \fs:=f_x(s,X_s,Y_s,Z_s)\nabla X_s
  +f_y(s,X_s,Y_s,Z_s)\nabla Y_s+f_z(s,X_s,Y_s,Z_s)\nabla Z_s,\\
  &\nabla \gs:=g_x(s,X_s,Y_s)\nabla X_s+g_y(s,X_s,Y_s)\nabla Y_s.
  \end{split}
\end{equation*}
Here we use subscripts to indicate partial differentiations.

(4) (Lemma 4.4 in \cite{Bachouch2016})~$\{Z_t\}_{0\leq t\leq T}$ has an a.s. continuous version which is given by
\begin{equation*}\label{90}
  Z_t=\nabla Y_t.
\end{equation*}
Furthermore, with the assumptions of the theorem and through similar estimation techniques for the variation equation for $Y_t$, we have
\begin{equation}
\begin{split}
  \quad E[(Z_t-Z_s)^2]\leq C(t-s), \quad | E[ Z_t - Z_s ] | \leq C(t-s),
  \end{split}
\end{equation}
for some positive constant $C$, independent of $t$.

\end{theorem}

\vspace{0.1em}

\section{Splitting up method and convergence analysis}
In this section, we introduce the  splitting up framework for BDSDE \eqref{1b} and show that our splitting up system  provides a first order approximation for the original BDSDE.
\subsection{Splitting up method}
Let $0=t_0<t_1<\cdots<t_N=T$ be an uniform partition of $[0,T]$ with partition size $\Delta t:=\f{T}{N}$, where $N$ is a positive integer. Denote  $\Delta W_i:=W_{t_{i+1}}-W_{t_i}$ and $\Delta B_i:=B_{t_{i+1}}-B_{t_i}$. The approximation $\yb(t)$ to the solution $Y_t$ of BDSDE \eqref{1b} is defined recursively on each time  interval $[t_i,\tp), i=1, \cdots, N-1$ as follows. Set  $Y_{N}(T)=\Phi (X_T)$.  First define $ \ya(t), t_i\leq t<t_{i+1}$,  to be the solution of the BSDE:
\begin{equation}\label{BSDE_interval}
  \ya(t) =\yatp+\iit\ffs\ds-\iit\z(s)\dws, \ (BSDE)
\end{equation}
Then  $\yb(t)$ is defined as the solution of the SDE:
\begin{equation}\label{SDE_interval}
  \yb(t) =\ya(t)+\iit\ggs\dbs.  \ (SDE)
\end{equation}
In this way, the approximation of BDSDE \eqref{1b} on subinterval $[t_i,\tp)$ is split into two steps. In the first step, we  solve the  BSDE \eqref{BSDE_interval}. In the second step, we use the solution  $\ya(t_i)$ of the BSDE at time $t_i$ as the terminal value at time $\tp$ and solve the SDE \eqref{SDE_interval} on  $[t_i,\tp)$. These implicit equations are solved using iterative techniques. Here, the solution $(\ya,\z)$ of the BSDE \eqref{BSDE_interval} plays the role of the intermediate solution before we incorporate the $d\ola{B}$ integral. Hence $(\tilde{Y}_i(t), \tilde{Z}_i(t))$ is $\mathcal{F}_{0,t}^W\vee\mathcal{F}_{\tp,T}^B$ measurable for any $t\in[t_i,\tp)$. On the other hand, the solution $Y_i(t)$ of the SDE is $\mathcal{F}_{0,\tp}^W\vee\mathcal{F}_{t,T}^B$ measurable.  We let $\z$ be our approximation for the solution $Z_t$ for $t \in [t_i,\tp)$. It's worthy noting that $Y_i$ incorporates the $d\ola{B}$ integral as a solution for SDE, and $\tilde{Z}_i$ incorporates the $d\ola{B}$ integral only through the variation relationship with $Y_{i+1}(\tp)$ at temporal grid points. Moreover, letting $t\rightarrow \tp-0$, we have
\begin{equation*}
  \lim\limits_{t\rightarrow \tp-0}\yb(t)=\yatp.
\end{equation*}
Therefore the approximate process $\bar{Y}_t:=\sum\limits_{i=0}^{N-1}Y_i(t)\,1_{[t_i,\tp)}(t)+\Phi(X_T)\,1_{T}(t)$ has continuous trajectories.

\subsection{Convergence analysis}
We now turn to the convergence analysis for the proposed splitting up system \eqref{BSDE_interval}-\eqref{SDE_interval} in approximating the BDSDE \cref{1b}. We first state the main result of our analysis which  shows that our splitting up system provides a first order mean square approximation for solution $Y_t$ and half order mean square approximation for solution $Z_t$.
\vspace{0em}

\begin{theorem}\label{main}
Assume that \textbf{(H1)-(H3)} hold. Then for sufficiently large  $N$, there exists a positive constant $C$, independent of  $\Delta t$ and  $X_0$, such that
   \begin{equation}\label{Theorem:main}
   \max\limits_{1\leq i\leq N}\left(E\big[\E[Y_i(t_i)]-Y_{t_i}\big]^2 \right)\leq C\Delta t^2, \quad \max\limits_{1\leq i\leq N}\left(E\big[\tilde{Z}_i(t_i)-Z_{t_i}\big]^2 \right)\leq C\Delta t,
   \end{equation}
where $\E[\cdot]$ denotes the conditional expectation over the $\sigma$-algebra $\mathcal{G}_{0,t_i}=\mathcal{F}^W_{0,t_i}\vee \mathcal{F}^B_{0,T}$.

\end{theorem}
To prove the theorem, we need several estimations concerning  the intermediate approximation $\ya$ and $\tilde{Z}_i$  given by \cref{BSDE_interval}.
\vspace{0.5em}

\begin{lemma}\label{th21}
Under the assumptions \textbf{(H1)-(H3)}, for any given interval $[t_i,\tp)$, there is a constant $C$, independent of $\Delta t$ and $X_0$, such that
   \begin{equation*}
   \begin{split}
      & \sup\limits_{t\in[t_i,t_{i+1})}E[(\ya(t)-\yatp)^2]\leq C\Delta t, ~~|\E[\ya(t)-\yatp]|\leq C\Delta t,\\
       & \sup\limits_{t\in[t_i,t_{i+1})}E[(\tilde{Z}_i(t)-\tilde{Z}_i(t_i))^2]\leq C\Delta t,~~|\E[\tilde{Z}_i(t)-\tilde{Z}_i(t_i)]|\leq C\Delta t.
   \end{split}
     \end{equation*}
\end{lemma}
\begin{proof}
The estimations in the lemma follow directly from \Cref{th1} with the special case $g\equiv0$.
\end{proof}

\vspace{0.5em}



\begin{lemma}\label{prop_36}
Under the assumptions \textbf{(H1)-(H3)}, for any given interval $[t_i,\tp)$, there is a constant $C$, independent of $\Delta t$ and $X_0$,  such that
   \begin{equation*}
   \sup\limits_{t\in[t_i,t_{i+1})}\big(E[ (\ya(t)-Y_{t})^2]\big)\leq CE[(E_{t_{i+1}}[\yatp]-Y_{\tp})^2]+C\Delta t.
   \end{equation*}
\end{lemma}

\begin{proof}
Subtracting \cref{1b} for $t\in[t_i,\tp)$ from \cref{BSDE_interval}, and taking the conditional expectation $E[\cdot|\mathcal{G}_{0,t_{i+1}}]$ gives
\begin{equation}\label{20}
  \begin{split}
  &\ya(t)-Y_t=E_{t_{i+1}}[\yatp]-Y_{\tp}-\iit [\z(s)-Z_s]\dws\\
  &+\iit[\ffs-\fs]\ds-\iit \gs\dbs.
  \end{split}
\end{equation}
Note that $\ya(t)$ is $\mathcal{F}_{0,t}^W\vee \mathcal{F}_{\tp,T}^B$ measurable for $t<\tp$, thus it is $\mathcal{F}_{0,t}^W\vee \mathcal{F}_{t,T}^B$ measurable, i.e.~$\mathcal{G}_t$ measurable. Applying the generalized It\^{o}'s Lemma (see Lemma 1.3 in \cite{pardoux1994backward}) to $|\ya(t)-Y_t|^2$ and taking the expectation, we have, using Young's inequality with $\varepsilon=\f{1}{2L}$, and assumption \textbf{(H3)},
\begin{equation*}
   \begin{split}
       & E|\ya(t)-Y_t|^2+\iit E|\z(s)-Z_s|^2\ds\\
       &\leq E|E_{t_{i+1}}[\yatp]-Y_{\tp}|^2+2\iit L(E|X_s|^2+E|Y_s|^2)\ds\\
       &+2\iit E|g(s,0,0)|^2\ds+(2L+\f{1}{2})\iit E|\ya(s)-Y_s|^2\ds\\
       &+\f{1}{2}\iit E|\z(s)-Z_s|^2\ds.
   \end{split}
\end{equation*}
The desired result follows from  Gronwall's inequality, assumption \textbf{(H3)}, and \Cref{th1}.
\end{proof}

\begin{lemma}\label{th23}
Under assumptions \textbf{(H1)-(H3)}, for any given interval $[t_i,\tp)$, there exists a constant $C$, independent of $\Delta t$ and $X_0$,  such that
   \begin{equation}
   \sup\limits_{t\in[t_i,t_{i+1})}\big(E[(\yb(t)-Y_{t})^2]\big)\leq C E[(E_{t_{i+1}}[\yatp]-Y_{\tp})^2]+C\Delta t.
   \end{equation}
\end{lemma}
\begin{proof}
Note that
\begin{equation*}
  \yb(t)-Y_t=\yb(t)-\ya(t)+\ya(t)-Y_t.
\end{equation*}
This result is then  a direct consequence of \Cref{prop_36}, It\^{o}'s isometry, and the assumption \textbf{(H3)}.
\end{proof}

Combining \Cref{th1} and \Cref{th21}, using Young's inequality, we arrive an estimate on $E[(\z(t)-Z_t)^2]$.
\begin{lemma}\label{prop_38}
Under the assumptions \textbf{(H1)-(H3)}, for any given interval $[t_i,\tp)$, there exits constant a $C$, independent of $\Delta t$ and $X_0$, such that
    \begin{equation*}
   \sup\limits_{t\in[t_i,t_{i+1})}\big(E[(\z(t)-Z_t)^2]\big)\leq (1+\epsilon_0)E[(\tilde{Z}_i(t_{i+1}-0)-Z_{t_{i+1}})^2]+C\Delta t,
   \end{equation*}
for some suitable $\epsilon_0>0$.
\end{lemma}

\noindent {\bf Proof of \Cref{main}}:
The main ingredients of the proof are the  estimations for the errors $\yb(t_i)-Y_{t_i}$ and $\z(t_i)-Z_{t_i}$.  Once these estimations are obtained, the desired result of the theorem is the consequence of application of the discrete Gronwall inequality. \vspace{0.25em}

\noindent \textbf{Estimation for the error $\yb(t_i)-Y_{t_i}$.}\\
Subtracting \cref{1b} for $t\in[t_i,\tp)$ from \cref{SDE_interval} and substituting \cref{20} with result we have that for $t=t_i$
\begin{equation}\label{8}
\begin{split}
  &\yb(t_i)-Y_{t_i}=E_{t_{i+1}}[\yatp]-Y_{\tp}+\ii [ \ggs-\gs ]\dbs\\
  &+\ii [ \ffs-\fs] \ds-\ii [\z(s)-Z_s ]\dws.
  \end{split}
\end{equation}
To simplify notation in subsequent derivations, we shall use the following shorthand notation:
\begin{equation*}
\begin{split}
  &e^i_y:= E_{t_i}[Y_i(t_i)] - Y_{t_i},~e_z^{i+1}:=\tilde{Z}_i(t_{i+1}-0)-Z_{t_{i+1}},\\
  &\delta f^i(s):=\ffs-\fs,\\
  &\delta g^i(s):=\ggs-\gs.\\
  \end{split}
\end{equation*}

Taking  the conditional expectation $\E[\cdot]$ on both sides of the above yields
\begin{equation}\label{27}
  \begin{split}
E_{t_i}[Y_i(t_i)] - Y_{t_i}&=\E[E_{t_{i+1}}[\yatp]-Y_{\tp}]\\
&+\ii\E[ \delta f^i(s)]\ds+\ii \E[\delta g^i(s)]\dbs.
\end{split}
\end{equation}
Here we have used Fubini's theorem and the fact that $Y_{t_i}$ is $\mathcal{F}_{0,t_i}^W\vee \mathcal{F}_{0,T}^B$ measurable, i.e.~ $\mathcal{G}_{0,t_i}$ measurable.

Next we consider the mean square estimation for $e_y^i$. Square and then take the expectation on both sides of \cref{27} to obtain
\begin{equation}\label{e_y^2-1}
  \begin{split}
  &E[(e_y^i)^2]=E[(\E[e_y^{i+1}])^2]+E\Big[\Big(\ii\E[\delta f^i(s)]\ds+\ii \E[\delta g^i(s)]\dbs\Big)^2 \Big]\\
  & \quad +2E\bigg[\big(\E[e_y^{i+1}]\big)\cdot\bigg(\ii\E[\delta f^i(s)]\ds+ \ii\E[\delta g^i(s)]\dbs \bigg)\bigg].
\end{split}
\end{equation}
%
 Using the elementary inequality $(a+b)^2\leq 2(a^2+b^2)$ on
 \cref{e_y^2-1}, we have
\begin{equation}\label{91}
  \begin{split}
  &E[(e_y^i)^2]\leq I_1+I_2+I_3+I_4+I_5,
\end{split}
\end{equation}
where
\begin{equation*}
   \begin{split}
       & I_1:= E[(\E[e_y^{i+1}])^2],\\
       & I_2:=2E\Big[\big(\ii\E[\delta f^i(s)]\ds\big)^2\Big],\\
       & I_3:=2E\Big[\big(\ii\E[\delta g^i(s)]\dbs\big)^2\Big],\\
       & I_4:=2E\Big[\big(\E[e_y^{i+1}]\big)\cdot\big(\ii\E[\delta f^i(s)]\ds\big)\Big],\\
       & I_5:=2E\Big[\big(\E[e_y^{i+1}]\big)\cdot\big(\ii\E[\delta g^i(s)]\dbs\big)\Big].
   \end{split}
\end{equation*}
 By Cauchy's inequality, Jensen's inequality, and the assumptions \textbf{(H1)-(H3)}, we have
\begin{equation*}
   \begin{split}
       I_2&\leq 2 \Delta t E\Big[\ii ( \E[\delta f^i(s)])^2\ds\Big] \\
       &\leq 2 \Delta t \ii L\Big(E[(\ya(s)-Y_{s})^2 ]+E\big[(\z(s)-Z_s)^2\big]\Big)\ds.\\
   \end{split}
\end{equation*}
Then, from Lemma \ref{prop_36} and Lemma \ref{prop_38}, we get
\begin{equation}\label{14}
I_2 \leq C (\Delta t)^2  E[(e_y^{i+1})^2] + 2L(1+\epsilon_0)(\Delta t)^2 E\big[(e_z^{i+1})^2 \big] + O(\Delta t^3).
\end{equation}
Next we estimate $I_3$.  To simplify  the  presentation, we use abbreviated notations $\Theta_r=(r,X_r,Y_r)$ and $\tilde{\Theta}_r=(r,X_r,\yb(r))$, and use subscripts of function $g$ to indicate partial differentiations. We also use $\dd [X]_r$ to denote the quadratic variation of $X_r$, and $\dd [X,Y]_r$  the quadratic covariation of $X_r$ and $Y_r$.
In order to derive an estimation for $I_3$, we first apply the It\^{o}-Taylor expansions for $g(\Theta_s)$ and $g(\tilde{\Theta}_s)$ on interval $[s, t_{i+1}]$ to obtain
\begin{equation}\label{g1}
   \begin{split}
   &g(\Theta_{\tp})=g(\Theta_s)+\iis g_t(\Theta_r)\dr+\iis g_x(\Theta_r)\dd X_r+\iis g_y(\Theta_r)\dd Y_r\\
   &+\frac{1}{2}\iis g_{xx}(\Theta_r)\dd [X]_r+\frac{1}{2}\iis g_{yy}(\Theta_r)\dd [Y]_r+\iis g_{xy}(\Theta_r)\dd [X,Y]_r,
\end{split}
\end{equation}
and
\begin{equation}\label{g2}
   \begin{split}
   &g(\tilde{\Theta}_{\tp})=g(\tilde{\Theta}_s)+\iis g_t(\tilde{\Theta}_r)\dr+\iis g_x(\tilde{\Theta}_r)\dd X_r+\iis g_y(\tilde{\Theta}_r)\dd \yb(r)\\
   &+\frac{1}{2}\iis g_{xx}(\tilde{\Theta}_r)\dd [X]_r+\frac{1}{2}\iis g_{yy}(\tilde{\Theta}_r)\dd [\yb]_r+\iis g_{xy}(\tilde{\Theta}_r)\dd [X,\yb]_r.
\end{split}
\end{equation}
 Note that   $dX_r = dW_r$, it then follows from the generalized It\^{o}'s Lemma (see Lemma 1.3 in \cite{pardoux1994backward}) that
\begin{equation*}
   \begin{split}
       & \dd [X]_r=\dr, ~\dd [Y]_r=-g^2(\Theta_r)\dr+(Z_r)^2\dr,~\dd [X,Y]_r=Z_r\dr \\
       & \dd [\yb]_r=-g^2(\tilde{\Theta}_r)\dr+(\tilde{Z}_i(r))^2\dr,~\dd [X,\yb]_r=\tilde{Z}_i(r)\dr.
   \end{split}
\end{equation*}
Subtracting \cref{g1} from \cref{g2}, we have
\begin{equation}\label{dg}
   \begin{split}
   &g(\tilde{\Theta}_s)-g(\Theta_s)=g(\tilde{\Theta}_{\tp})-g(\Theta_{\tp})+\iis [(g_y\cdot g)(\tilde{\Theta}_r)-(g_y\cdot g)(\Theta_r)] \dbr\\
   &-\iis [(g_x(\tilde{\Theta}_r)+g_y(\tilde{\Theta}_r)\tilde{Z}_i(r))-(g_x(\Theta_r)+g_y(\Theta_r)Z_r)]\dd W_r+R_{g,Y}^{i}(s),
\end{split}
\end{equation}
where $R_{g,Y}^{i}$ contains all the $\iis\cdot\dr$ integrals:
\begin{equation*}
   \begin{split}
       &R_{g,Y}^{i}(s):=-\iis [g_t(\tilde{\Theta}_r)-g_t(\Theta_r)]\dr\\
&+\iis [g_y(\tilde{\Theta}_r)\ffr-g_y(\Theta_r)\fr]\dr\\
&-\frac{1}{2}\iis [g_{xx}(\tilde{\Theta}_r)+g_{yy}(\tilde{\Theta}_r)(\tilde{Z}_i(r))^2-g_{xx}(\Theta_r)-g_{yy}(\Theta_r)(Z_r)^2]\dr\\
&+\frac{1}{2}\iis [g_{yy}(\tilde{\Theta}_r)g^2(\tilde{\Theta}_r)-g_{yy}(\Theta_r)g^2(\Theta_r)]\dr\\
&-\iis [g_{xy}(\tilde{\Theta}_r)\tilde{Z}_i(r)-g_{xy}(\Theta_r)Z_r]\dr,
   \end{split}
\end{equation*}
and it's easy to see that $\sup\limits_{s\in[t_i,t_{i+1})}E[(R_{g,Y}^{i}(s))^2] = O(\Delta t^2)$.

Taking the conditional expectation $\E[\cdot]$ on both sides of \cref{dg}, we have
\begin{equation*}
   \begin{split}
        \E[\delta g^i(s)]& =\E[\delta g^i(\tp)] +\E\bigg[\iis\delta (g_yg)^i(r)\dbr\bigg]+\E[R_{g,Y}^{i}(s)],\\
   \end{split}
\end{equation*}
where $\delta (g_yg)^i(r):=(g_y\cdot g)(\tilde{\Theta}_r)-(g_y\cdot g)(\Theta_r)$. By the above estimation,  It\^{o}'s isometry, the elementary inequality $(a+b+c)^2\leq 3(a^2+b^2+c^2)$, and Jensen's inequality, we obtain
\begin{equation}
   \begin{split}
       I_3&=2E\bigg(\ii(\E[\delta g^i(s)])^2\ds\bigg)\\
       &\leq 6\ii\bigg(E\big[\big(\E[\delta g^i(\tp)]\big)^2\big]+\iis E[(\E[\delta (g_yg)^i(r)])^2]\dr+E[(R_{g,Y}^{i}(s))^2]\bigg)\ds.
   \end{split}
\end{equation}
Using Lemma \ref{th23}, the assumptions \textbf{(H1)-(H3)} and Jensen's inequality, we have
\begin{equation}\label{18}
   \begin{split}
       I_3&\leq 6|g_y|_{\infty}^2\Delta t E[(e_y^{i+1})^2]+C \Delta t ^2 E[(e_y^{i+1})^2]+ O(\Delta t ^3).\\
   \end{split}
\end{equation}

We now turn to the estimation of  $I_4$.  First we decompose $\delta f^i(s)$, which is the abbreviation for $\ffs-\fs$, into three parts to write $I_4$ as
\begin{equation*}\label{71}
\begin{split}
I_4=2E\bigg[\E[e_y^{i+1}]\cdot\ii\big(\E[\delta f^{i,a}]+\E[\delta f^{i,b}]+\E[\delta f^{i,c}]  \big)\ds\bigg],\\
  \end{split}
\end{equation*}
where
\begin{equation*}
   \begin{split}
   &\delta f^{i,a}:=\ffs-f(\tp,X_{\tp},\yatp,\tilde{Z}_{i}(t_{i+1}-0)),\\
   &\delta f^{i,b}:=f(\tp,X_{\tp},\yatp,\tilde{Z}_{i}(t_{i+1}-0))-f(\tp,X_{\tp},Y_{\tp},Z_{t_{i+1}}),\\
&\delta f^{i,c}:=f(\tp,X_{\tp},Y_{\tp},Z_{t_{i+1}})-\fs.
   \end{split}
\end{equation*}
By It\^{o}-Taylor expansion and Lemma 3.2, we see that $\E[\delta f^{i,a}]=O(\Delta t)$. Hence
\begin{equation*}
  2E\bigg[\E[e_y^{i+1}]\cdot\ii\E[\delta f^{i,a}]\ds\bigg]\leq \f{\Delta t}{3} E\big[(e_y^{i+1})^2\big]  + O(\Delta t^3).
\end{equation*}
On the other hand, by Theorem \ref{th1}, we have $E[\delta f^{i,c}]=O(\Delta t)$. Therefore, it follows from the properties of conditional expectations and Young’s inequality that
\begin{equation}\label{2}
   \begin{split}
       & 2E\bigg[\E[e_y^{i+1}]\cdot\ii\big(\E[\delta f^{i,c}]  \big)\ds\bigg]=2E\bigg[E\big[\E[e_y^{i+1}]\cdot\ii\big(\E[\delta f^{i,c}]  \big)\ds|\mathcal{F}^B_{t_i,t_{i+1}}\big]\bigg] \\
       & =2E[e_y^{i+1}]\ii E[\delta f^{i,c}]\ds\leq \f{\Delta t}{3} E\big[(e_y^{i+1})^2\big]  + O(\Delta t^3).
   \end{split}
\end{equation}
Putting the above estimates together, then using the assumptions \textbf{(H1)}-\textbf{(H3)} and Young’s inequality, we obtain
\begin{equation}\label{29}
\begin{aligned}
I_4 
& \leq {\Delta t} E\big[(\E[e_y^{i+1}])^2\big] + 6|f_y|_{\infty}^2\Delta t E[(e_y^{i+1})^2]+ 6|f_z|_{\infty}^2\Delta t E[(\E[e_z^{i+1}] )^2] + O(\Delta t^3)\\
&=(1 + 6|f_y|_{\infty}^2)\Delta t E[(e_y^{i+1})^2]+6|f_z|_{\infty}^2 \Delta t E\big[(\E[e_z^{i+1}] )^2 \big]+O(\Delta t^3).
\end{aligned}
\end{equation}
To estimate the last term, we use an argument similar to \eqref{2}.
\begin{equation*}
   \begin{split}
       I_5=&2E\Big[E[\E[e_y^{i+1}]\cdot\ii\E[\delta g^i(t_{i+1})]\dbs|\mathcal{F}^B_{t_i,t_{i+1}}]\Big]\\
       &+2E\Big[E[\E[e_y^{i+1}]\cdot\ii(\E[\delta g^i(s)-\delta g^i(t_{i+1})])\dbs|\mathcal{F}^B_{t_i,t_{i+1}}]\Big] \\
       & =2E\Big[\E[e_y^{i+1}]\E[\delta g^i(t_{i+1})]\Big]E[\dbi]+2E[e_y^{i+1}]\cdot E\bigg[\ii(\E[\delta g^i(s)-\delta g^i(t_{i+1})])\dbs\bigg]=0.
   \end{split}
\end{equation*}

Substituting the above estimations \cref{14,18,29} into \cref{91}, we have
\begin{equation}\label{Estimation:Y}
  \begin{split}
  E[(e_y^i)^2]\leq&   E[(\E[e_y^{i+1}])^2]+C_y^1 \Delta t E\big[(e_y^{i+1} )^2 \big]+6|f_z|_{\infty}^2\Delta t E\big[(\E[e_z^{i+1}] )^2 \big]\\
  &+ C \Delta t^2 E[(e_y^{i+1})^2] + 2L(1+\epsilon_0) \Delta t^2 E\big[(\E[e_z^{i+1}] )^2 \big] +O(\Delta t^3),\\
\end{split}
\end{equation}
where $C_y^1 : = 1 + 6|f_y|_{\infty}^2+6|g_y|_{\infty}^2$ is a constant.
\vspace{0.5em}

\noindent \textbf{Estimation for the error $\z(t_i)-Z_{t_i}$.}\\
In order to derive an estimation for $\z(t_i)-Z_{t_i}$, we subtract \cref{1b} for $t\in[t_i,\tp)$ from \cref{20}, let $t = t_i$, multiply both sides of  the resulting equation by $\dwi$, and then take the conditional expectation $\E[\cdot]$ to obtain
\begin{equation}\label{6}
\begin{split}
  &  \ii \E[\tilde{Z}^i(s)-Z_s]\ds =\E[e_y^{i+1}\dwi] +\ii\E[\delta f^i(s)\dwi]\ds \\
  & \hspace{6em}  - \ii \E\Big[g(s,X_s,Y_s)\dwi \Big]\dbs ,
\end{split}
\end{equation}
where we have used the  Fubini's theorem and the fact that $\E[( \tilde{Y}_i(t_{i})-Y_{t_i}) \Delta W_i]=0$.
Rewrite $\E[\tilde{Z}^i(s)-Z_s]$ as
\begin{equation*}
    \E[\tilde{Z}_i(s)- \tilde{Z}_i(t_{i+1}-0)] +\E[ \tilde{Z}_i(t_{i+1}-0) - Z_{t_{i+1}}] +\E[ Z_{t_{i+1}} - Z_s ].
\end{equation*}
Using \Cref{th1} (4), we have
\begin{equation}\label{difference_W}
\begin{split}
  &\Delta t \E[e_z^{i+1}] = \ \E[e_y^{i+1}\dwi]+\ii\E[\delta f^i(s)\dwi]\ds \\
  &-\ii \E[g(\Theta_s) \Delta W_i ]\dbs+ \int_{t_i}^{t_{i+1}}\int_{s}^{t_{i+1}}   \E[\nabla g(\Theta_r)] \dbr\ds + O(\Delta t^2),
  \end{split}
\end{equation}
where  $\Theta_r=(r,X_r,Y_r)$ . Similar to the argument in \cref{g1} ( notice that here we apply the It\^o formula on interval $[t_i, s]$ instead of $[s, t_{i+1}]$ in \cref{g1} ), we obtain
$$
g(\Theta_s) = g(\Theta_{t_i}) + \int_{t_i}^{s} \big(g_x(\Theta_r) + g_y(\Theta) Z_r \big)dW_r - \int_{t_i}^{s} (g_y\cdot g)(\Theta_r) \dbr + R_g^i,
$$
where all the $\int_{t_i}^{s} \cdot dr$ terms are included in $R_g^i$. The above equation leads to
\begin{equation*}
\begin{aligned}
& \int_{t_i}^{t_{i+1}}  \E\Big[g(\Theta_s) \Delta W_i \Big] \dbs= \int_{t_i}^{t_{i+1}}  \int_{t_i}^{s} \E\big[g_x(\Theta_r) + g_y(\Theta_r) Z_r \big] \dr \dbs \\
& \hspace{-1em} - \int_{t_i}^{t_{i+1}}  \int_{t_i}^{s} \E[ (g_y\cdot g)(\Theta_r) \Delta W_i]  \dbr \dbs + \int_{t_i}^{t_{i+1}}  \E\Big[R_g^i \Delta W_i \Big] \dbs.
\end{aligned}
\end{equation*}
Here we have used the fact $\E[g(\Theta_{t_i}) \Delta W_i]  = 0$. Decompose the second term on the right hand side of the above equation into two terms
$$
\begin{aligned}
& \int_{t_i}^{t_{i+1}}  \int_{t_i}^{s} \E[ (g_y\cdot g)(\Theta_{t_i}) \Delta W_i]  \dbr \dbs\\
& \qquad + \int_{t_i}^{t_{i+1}}  \int_{t_i}^{s} \E[ \big( (g_y\cdot g)(\Theta_r) - (g_y\cdot g)(\Theta_{t_i}) \big)\Delta W_i] \dbr \dbs.
\end{aligned}
$$
By properties of conditional expectations, the first part is 0. Therefore
\begin{equation*}\label{gdW}
\begin{split}
 &\int_{t_i}^{t_{i+1}}  \E\Big[g(\Theta_s) \Delta W_i \Big] \dbs= \int_{t_i}^{t_{i+1}}  \int_{t_i}^{s} \E\big[g_x(\Theta_r) + g_y(\Theta_r) Z_r \big] \dr \dbs\\
 &-\int_{t_i}^{t_{i+1}}  \int_{t_i}^{s} \E[ \big( (g_y\cdot g)(\Theta_r) - (g_y\cdot g)(\Theta_{t_i}) \big)\Delta W_i] \dbr \dbs+ \int_{t_i}^{t_{i+1}}  \E\Big[R_g^i \Delta W_i \Big] \dbs.
\end{split}
\end{equation*}
Putting this into \cref{difference_W}, and noting that $\nabla g (\Theta_r) = g_x(\Theta_r) + g_y(\Theta_r) Z_r$ (as defined in \Cref{th1}), we obtain
%
\begin{equation}\label{difference_W:modified}
\begin{split}
  &\Delta t \E[e_z^{i+1}] =  \E[e_y^{i+1}\dwi]+\ii\E[\delta f^i(s)\dwi]\ds+R_z^{i+1},\\
  \end{split}
  \end{equation}
where
\begin{equation*}
\begin{split}
  & R_z^{i+1}=-\int_{t_i}^{t_{i+1}}  \int_{t_i}^{s} \E\big[\nabla g(\Theta_r)-\nabla g(\Theta_{t_i}) \big] \dr \dbs\\
 &+\int_{t_i}^{t_{i+1}}  \int_{t_i}^{s} \E[ \big( (g_y\cdot g)(\Theta_r) - (g_y\cdot g)(\Theta_{t_i}) \big)\Delta W_i] \dbr \dbs- \int_{t_i}^{t_{i+1}}  \E\Big[R_g^i \Delta W_i \Big] \dbs\\
 &+ \int_{t_i}^{t_{i+1}}\int_{s}^{t_{i+1}}   \E[\nabla g(\Theta_r)-\nabla g(\Theta_{t_i})] \dbr\ds + O(\Delta t^2).
  \end{split}
\end{equation*}
It is easy to check that $E[(R_z^{i+1})^2]=O(\Delta t^4)$. Similar to the discussions for the $e_y^{i}$ error, we square \cref{difference_W:modified}, and take expectation to get
\begin{equation}\label{difference_W^2}
\begin{split}
  &\Delta t^2 E\big[(\E[e_z^{i+1}] )^2 \big] 
  \leq E\big[ \big( \E[e_y^{i+1}\dwi] \big)^2 \big] + 2 E\bigg[\bigg(\ii\E[\delta f^i(s)\dwi]\ds\bigg)^2\bigg] \\
  & + 2  E\Big[ \E[e_y^{i+1}\dwi] \cdot\ii\E[\delta f^i(s)\dwi]\ds \Big]+ E\Big[ \E[e_y^{i+1}\dwi] R_z^{i+1} \Big] + O(\Delta t^4) .
  \end{split}
\end{equation}
For the cross product terms in the above equation, we use Young's inequality to get
$$
\begin{aligned}
& 2  E\Big[ \E[e_y^{i+1}\dwi] \ii\E[\delta f^i(s)\dwi]\ds \Big] \\
\leq & \ \f{L \Delta t}{\epsilon_1} E\big[\big( \E[e_y^{i+1}\dwi] \big)^2 \big] + \f{\epsilon_1}{L \Delta t} E\Big[ \Big(\ii\E[\delta f^i(s)\dwi]\ds\Big)^2\Big]\\
\leq & \ \f{L \Delta t^2}{\epsilon_1} E\big[\big( e_y^{i+1} \big)^2 \big] + \epsilon_1 \Delta t^2 E\Big[ (C\E\big[ e_y^{i+1}])^2 + (1+\epsilon_0)( \E[e_z^{i+1}])^2\big]\Big]+O(\Delta t^4)
\end{aligned}
$$
and
$$E\Big[ \E[e_y^{i+1}\dwi]R_z^{i+1} \Big] \leq \epsilon_2 E\Big[ \big(\E[e_y^{i+1}\dwi ] \big)^2 \Big]+ O(\Delta t^4). $$
From the assumption \textbf{(H3)}, Lemma \ref{prop_36} and Lemma \ref{prop_38}, and the above estimations, we have
\begin{equation*}
\begin{split}
  \Delta t^2 E\big[ ( \E[e_z^{i+1}])^2 \big]   & \leq E\big[ \big( \E[e_y^{i+1}\dwi] \big)^2 \big] + C \Delta t^3  E[(e_y^{i+1})^2] + 2L(1+\epsilon_0)\Delta t^3 E\big[\big(\E[e_z^{i+1}]\big)^2\big]  \\
  & \quad + \f{L \Delta t^2}{\epsilon_1} E\big[\big( e_y^{i+1} \big)^2 \big] + C\epsilon_1 \Delta t^2 E[ (e_y^{i+1})^2] + \epsilon_1(1+\epsilon_0) \Delta t^2 E[( \E[e_z^{i+1}])^2]\\
  & \quad  + \epsilon_2 E\Big[ \big(\E[e_y^{i+1}\dwi ] \big)^2 \Big]+ O(\Delta t^4).
  \end{split}
\end{equation*}
Set $\f{1}{1+\epsilon_2} < 1$. Dividing both sides of the above estimate by $\Delta t (1 + \epsilon_2)$, and noting that $\big(\E[ e_y^{i+1}\dwi ]\big)^2\leq \Delta t \Big(\E[(e_y^{i+1})^2]-(\E[e_y^{i+1}])^2\Big)$ ,  we obtain
\begin{equation}\label{Estimation:Z}
\begin{split}
  \f{\Delta t}{1 + \epsilon_2} E\big[ (\E[e_z^{i+1}])^2 \big]   & \leq E\big[\E[(e_y^{i+1})^2]-(\E[e_y^{i+1}])^2\big] + \big( \f{L}{\epsilon_1} +C\epsilon_1 \big)\Delta t E\big[\big( e_y^{i+1} \big)^2 \big]\\
  &  \quad + \epsilon_1(1+\epsilon_0) \Delta t E\Big[ \E[(e_z^{i+1})^2\big]\Big] +  C \Delta t^2  E[(e_y^{i+1})^2] \\
  & \quad + 2L(1+\epsilon_0)\Delta t^2E\big[\big(\E[e_z^{i+1}]\big)^2\big] + O(\Delta t^3).
  \end{split}
\end{equation}
which is the desired estimate for $e_z^{i+1}$.

\vspace{0.5em}

Now  we use the above estimates, \cref{Estimation:Y} for $\yb(t_i)-Y_{t_i}$, and \cref{Estimation:Z} for $\z(t_i)-Z_{t_i}$,  to derive the error estimate of the theorem.
First  we combine  \cref{Estimation:Y} and \cref{Estimation:Z} to obtain
\begin{equation}\label{Y+Z}
\begin{aligned}
& E[(e_y^i)^2] +   \f{\Delta t}{1 + \epsilon_2} E\big[ ( \E[e_z^{i+1}])^2 \big]  \\
\leq & \ E[ (e_y^{i+1})^2] + \big(C_y^1 + \f{L}{\epsilon_1} + C\epsilon_1 \big) \Delta t E[(e_y^{i+1})^2]+ \big(6|f_z|^2_\infty + \epsilon_1(1+\epsilon_0)\big) \Delta t E\big[( \E[e_z^{i+1}])^2 \big] \\
& \ + C\Delta t^2 E[(e_y^{i+1})^2] + 4L(1+\epsilon_0) \Delta t^2 E\big[( \E[e_z^{i+1}])^2\big] + O(\Delta t^3).
\end{aligned}
\end{equation}
Next we properly choose  $\epsilon_0$, $\epsilon_1$ and $\epsilon_2$ to control the $E\big[ ( \E[e_z^{i+1}])^2\big]$ terms on the right hand side of \cref{Y+Z}. Specifically, we move all the $E\big[( \E[e_z^{i+1}])^2\big]$ terms to the left hand side and get
\begin{equation}\label{Estimate:Y+Z}
\begin{aligned}
 E[(e_y^i)^2] +  C_z \Delta t E\big[ ( \E[e_z^{i+1}])^2 \big]  \leq E[ (e_y^{i+1})^2] + C_y \Delta t E[(e_y^{i+1})^2] + O(\Delta t^3)
\end{aligned}
\end{equation}
where $C_y = C_y^1 + \f{L}{\epsilon_1} + C\epsilon_1 + C\Delta t$, $C_z = \f{1}{1+\epsilon_2} - 6|f_z|^2_\infty  - \epsilon_1(1+\epsilon_0) - 4L(1+\epsilon_0) \Delta t$. Now we choose constants  $\epsilon_0$, $\epsilon_1$ and $\epsilon_2$ sufficiently small  such that $\f{1}{1+\epsilon_2} - 6|f_z|^2_\infty  - \epsilon_1(1+\epsilon_0) > 0$.  When the temporal step size $\Delta t$ is chosen such that   $\Delta t < (\f{1}{1+\epsilon_2} - 6|f_z|^2_\infty  - \epsilon_1(1+\epsilon_0) )/ 4L(1+\epsilon_0)$, we have $C_z > 0$. As a result, the estimate \cref{Estimate:Y+Z} becomes
\begin{equation*}
 E[(e_y^i)^2]   \leq E[ (e_y^{i+1})^2] + C_y \Delta t E[(e_y^{i+1})^2] + O(\Delta t^3),
\end{equation*}
which gives
\begin{equation}\label{Analysis:Y}
\max\limits_{1\leq i\leq N}\left(E[(e_y^i)^2] \right)\leq C\Delta t^2
\end{equation}
according to the discrete Gronwall 's inequality. This is the first part of \cref{Theorem:main}.

For the second part of  \cref{Theorem:main}, we substitute  \cref{Analysis:Y} into \cref{Estimate:Y+Z} to obtain
\begin{equation}\label{Analysis:Z}
 \max\limits_{1\leq i\leq N-1} E\big[ ( \E[e_z^{i+1}])^2 \big] \leq C \Delta t.
 \end{equation}
Since $\tilde{Z}_i(t_{i}) - Z_{t_{i}}$ is $\mathcal{F}_{0,t_i}^W\vee \mathcal{F}_{0,T}^B$ measurable, i.e.~$\mathcal{G}_{0,t_i}$ measurable, we have
\begin{equation*}
   \begin{split}
&  \max\limits_{1\leq i\leq N-1}  E\big[(\tilde{Z}_i(t_{i}) - Z_{t_{i}})^2\big] =  \max\limits_{1\leq i\leq N-1}  E\big[(\E[\tilde{Z}_i(t_{i}) - Z_{t_{i}}])^2\big] \\
\leq&   \max\limits_{1\leq i\leq N-1}  3 E\Big[(\E[\tilde{Z}_i(t_{i}) - \tilde{Z}_i(t_{i+1}-0)])^2 + (\E[e_z^{i+1}])^2   +(\E[Z_{t_{i+1}} - Z_{t_{i}}])^2\Big].
   \end{split}
\end{equation*}
Then the second part of the theorem follows from \Cref{th1}, \Cref{th21}, and \cref{Analysis:Z}.
$\Box$

\vspace{0.1em}

\section{A first order splitting up  scheme}
In this section, we discretize the BSDE \eqref{BSDE_interval} and SDE \eqref{SDE_interval} in the splitting up system to obtain a first order splitting up numerical scheme.

First we define an approximation for $X$  by  $X^0=X_0$, and
\begin{equation}\label{86}
\begin{split}
  X^{i+1}=X^{i}+\dwi, \quad i=1,\cdots,N.
  \end{split}
\end{equation}
It is easy to see that for  any $t\in[t_i,\tp)$ and $i=1\cdots,N$, there exists a positive constant $C$, independent of $X_0$ and $\Delta t$, such that
\begin{equation}
  E\bigg[|X^{i+1}-X_t|^2+|X_t-X^{i}|^2\bigg]\leq C\Delta t.
\end{equation}
To obtain a first order splitting up scheme, we use the explicit Euler scheme to approximate  the BSDE \cref{BSDE_interval} and the  Milstein scheme to approximate  the SDE \cref{SDE_interval}. The resulting  algorithm is given as follows.
\begin{equation}\label{Numerical_Scheme}
\begin{aligned}
&H^{i+1}=Y^{i+1}+\Delta t f(t_{i+1}, X^{i+1}, Y^{i+1}, Z^{i+1}),&(a)\\
    &\tilde{Y}^{i} =\E[H^{i+1}],&(b)\\
  &Z^i = \f{1}{\Delta t} \E[H^{i+1} \Delta W_i ],&(c)\\
  &Y^i = \tilde{Y}^{i} + \E[G^{i+1}(\tilde{Y}^{i})],&(d)\\
  \end{aligned}
\end{equation}
where for any $\mathcal{F}_{t_{i+1},T}^B$ measurable random variable, $G^{i+1}(\xi)$ is defined by
\begin{equation}\label{Def:G}
 G^{i+1}(\xi) :=  g(t_{i+1}, X^{i+1}, \xi) \Delta B_i + (g_y \cdot g) (t_{i+1}, X^{i+1}, \xi) \f{1}{2} (\Delta B_i^2 - \Delta t).
 \end{equation}
Note that in \cref{Numerical_Scheme}, $\tilde{Y}^{i}$ is an approximation for $\tilde{Y}_i(t_i)$, $Y^i$ is an approximation for $\E[Y_i(t_i)]$,  and $Z^i$ is an approximation for $\tilde{Z}_i(t_i)$. Apparently, $Y^i$ is also an approximation for $Y_{t_i}$ and $Z^i$ is also an approximation for $Z_{t_i}$.

In order to show that $Y^i$ is a first order numerical approximation for $Y_t$ and $Z^i$ is a half order numerical approximation for $Z_t$, we first show that $Y^i$ is a first order approximation for $\E[Y_i(t_i)]$ and $Z^i$ is a half order approximation for $\tilde{Z}_i(t_i)$. Then, the first order convergence rate and half order convergence rate of our numerical schemes in approximating $Y_{t}$ and $Z_{t}$, respectively, is  arrived  as a direct consequence of   \Cref{main}.

\vspace{0.1em}

\begin{theorem}\label{discrete-split}
Assume that assumptions \textbf{(H1)-(H3)} hold. Then there exists a positive constant $C$, independent of $\Delta t$ and $X_0$,  such that
   \begin{equation}
   \max\limits_{1\leq i\leq N} E\big[(Y^i- \E[Y_i(t_i)])^2 + \Delta t (Z^i-\tilde{Z}_i(t_i))^2 \big] \leq C \Delta t^2.
   \end{equation}
\end{theorem}

\begin{proof}
Set $t = t_i$, take conditional expectation on both sides of \cref{SDE_interval}, and then subtract the result from \cref{Numerical_Scheme} (d) to get
\begin{equation}\label{SDE_difference}
Y^i - \E[Y_i(t_i)] = \tilde{Y}^{i} - \tilde{Y}_i (t_i) + \E[G^{i+1}(\tilde{Y}^{i}) - G^{i+1}(\tilde{Y}_i (t_i))] + R_{g}^i,
\end{equation}
where
\begin{equation*}
   \begin{split}
     R_{g}^i = & \int_{t_i}^{t_{i+1}} \E[g(\tilde{\Theta}_s)] \dbs -  \E[G^{i+1}(\tilde{Y}_i (t_i))]. \\
   \end{split}
\end{equation*}
It is easy to verify that $E[(R_{g}^i)^2]  = O(\Delta t^3 )$. As in  Section 3, denote
\begin{equation*}
  \hat{e}_y^i =Y^i - \E[Y_i(t_i)],~\text{and}~\hat{e}_z^i=Z^i-\tilde{Z}_i(t_i).
\end{equation*}
Squaring \cref{SDE_difference} and taking the expectation, we obtain
\begin{equation}\label{square:difference}
\begin{aligned}
E[( \hat{e}_y^i )^2] =& \ E\big[(\tilde{Y}^{i} - \tilde{Y}_i (t_i)  )^2 \big]  + E \big[ \Big(\E[G^{i+1}(\tilde{Y}^{i}) - G^{i+1}(\tilde{Y}_i (t_i))] + R_{g}^i \Big)^2 \big] \\
& \ + 2   E\Big[\big( \tilde{Y}^{i} - \tilde{Y}_i (t_i) \big) \Big(\E[G^{i+1}(\tilde{Y}^{i}) - G^{i+1}(\tilde{Y}_i (t_i))] + R_{g}^i \Big) \Big].
\end{aligned}
\end{equation}
The last term above is 0, which can be proved by an argument similar to \eqref{2}. Next we estimate  $\tilde{Y}^{i} - \tilde{Y}_i (t_i)$ and $\E[G^{i+1}(\tilde{Y}^{i}) - G^{i+1}(\tilde{Y}_i (t_i))] $ in \cref{square:difference}. By the definition of $\hat{G}$ in \cref{Def:G}, we have
\begin{equation}\label{estimation:G}
\begin{aligned}
& E\big[\big(G^{i+1}(\tilde{Y}^{i}) - G^{i+1}(\tilde{Y}_i (t_i)) \big)^2\big] \leq C_G \Delta t  E[(\tilde{Y}^{i} - \tilde{Y}_i (t_i))^2],
\end{aligned}
\end{equation}
where $C_G$ is a constant independent of $X_0$ and $\Delta t.$ Therefore, it suffices to estimate $E[(\tilde{Y}^{i} - \tilde{Y}_i (t_i))^2]$. To this end, we take the  conditional expectation $\E[\cdot]$ on both sides of  \cref{20} and subtract it from \cref{Numerical_Scheme} (b) to obtain
\begin{equation}\label{BSDE_difference}
\begin{aligned}
\tilde{Y}^{i} - \tilde{Y}_i (t_i) = & \E[\hat{e}_y^{i+1}] + \Delta t \E\Big[f(\Pi^{i+1})- f(\tilde{\Pi}_{\tp})\Big]+ R_{f}^i,
\end{aligned}
\end{equation}
where $ R_{f}^i =  \int_{t_i}^{t_{i+1}} \E\big[f(\tilde{\Pi}_{i+1}) - \ffs\big]\ds$ is the truncation error, and $R_f^i = O(\Delta t^2)$ (for notational simplicity, we denote $\Pi^{i+1}:=(t_{i+1}, X^{i+1}, Y^{i+1}, Z^{i+1})$ and $\tilde{\Pi}_{i+1}:=(\tp,X_{\tp},\yatp,\z(\tp-0))$).
Squaring  both sides of the above, and then taking the  expectation, we have
\begin{equation}\label{BSDE_difference-square}
\begin{aligned}
E\big[ \big( \tilde{Y}^{i} - \tilde{Y}_i (t_i) \big)^2 \big] =&  E\big[ \big( \E[\hat{e}_y^{i+1}] \big)^2 \big] + E\Big[ \Big( \Delta t \E\Big[f(\Pi^{i+1})- f(\tilde{\Pi}_{\tp})\Big]+ R_{f}^i \Big)^2\Big] \\
& + 2 E\Big[  \E[\hat{e}_y^{i+1}]  \Big( \Delta t \E\Big[f(\Pi^{i+1})- f(\tilde{\Pi}_{\tp})\Big]+ R_{f}^i \Big)\Big].
\end{aligned}
\end{equation}
Using similar arguments as \cref{14,18,29}, we obtain
\begin{equation}\label{square:Y-estimation}
\begin{aligned}
E[( \hat{e}_y^i)^2] \leq& \ E\big[ \big( \E[\hat{e}_y^{i+1} ] \big)^2 \big]  + \Delta t \hat{C}_y^1 E\big[(  \hat{e}_y^{i+1}  )^2 \big] \\
& \quad + \epsilon_1 \Delta t E\big[(  \hat{e}_z^{i+1}  )^2 \big] + \hat{C}_y^2 \Delta t^2 E\big[(  \hat{e}_z^{i+1}  )^2 \big] + O(\Delta t^3),
\end{aligned}
\end{equation}
where $\hat{C}_y^1$ and $\hat{C}_y^2$ are constants independent of $X_0$ and $\Delta t$, and $\epsilon_1$ is a constant to be specified later.

\vspace{0.25em}

To estimate  $\hat{e}_z^i$, we multiply $\Delta W_i$ on both sides of \cref{BSDE_interval}, take conditional expectation $\E[\cdot]$, and subtract it from  \cref{Numerical_Scheme} (c) to obtain
\begin{equation}\label{BSDE_dW_difference}
\begin{aligned}
\Delta t \big( \hat{e}_z^i \big) =& \E[\big(\hat{e}_y^{i+1} \big) \Delta W_i] + \Delta t \E\Big[ \Big( f(\Pi^{i+1})- f(\tilde{\Pi}_{\tp}) \Big) \Delta W_i \Big]+ R_{fW}^i,
\end{aligned}
\end{equation}
where
$$
\begin{aligned}
R_{fW}^i =&  \int_{t_i}^{t_{i+1}} \E\Big[\big( f(\tilde{\Pi}_{\tp}) - \ffs \big) \Delta W_i \Big]\ds \\
& + \int_{t_i}^{t_{i+1}} \E[\tilde{Z}_i(s) - \tilde{Z}_i(t_i) ]\ds.
\end{aligned}
$$
It follows from \Cref{th21} that $R_{fW}^i = O(\Delta t^2)$. Squaring  both sides of \cref{BSDE_dW_difference},  taking the expectation, and using similar analysis techniques as in the proof of \Cref{main}, we  derive that
\begin{equation}\label{square:Z-estimation}
\begin{aligned}
\f{\Delta t}{1 + \epsilon} E[( \hat{e}_z^i)^2] \leq& \ E\Big[ \E[\big(\hat{e}_y^{i+1} \big)^2] - (\E[\hat{e}_y^{i+1}])^2 \Big]  + \Delta t \hat{C}_z^1 E\big[(  \hat{e}_y^{i+1}  )^2 \big] \\
& \quad + \epsilon_2 \Delta t E\big[(  \hat{e}_z^{i+1}  )^2 \big] + \hat{C}_z^2 \Delta t^2 E\big[(  \hat{e}_z^{i+1}  )^2 \big]  + O(\Delta t^3),
\end{aligned}
\end{equation}
where $\hat{C}_z^1$, $\hat{C}_z^2$ are constants independent of $\Delta$ and $X_0$, and $\epsilon_2$ is a constant that will be determined later.

Finally, we add \cref{square:Y-estimation} and \cref{square:Z-estimation} to obtain
\begin{equation}\label{square:Y+Z-estimation}
\begin{aligned}
E[( \hat{e}_y^i)^2]  + & \f{\Delta t}{1 + \epsilon} E[( \hat{e}_z^i)^2] \leq \ E\big[ (\hat{e}_y^{i+1} )^2\big]  +  ( \epsilon_1 + \epsilon_2) \Delta t E\big[(  \hat{e}_z^{i+1}  )^2 \big]  \\
& + \Delta t ( \hat{C}_y^1+\hat{C}_z^1) E\big[(  \hat{e}_y^{i+1}  )^2 \big] + (\hat{C}_y^2+\hat{C}_z^2) \Delta t^2 E\big[(  \hat{e}_z^{i+1}  )^2 \big]  + O(\Delta t^3).
\end{aligned}
\end{equation}
Choosing  $\epsilon$, $\epsilon_1$ and $\epsilon_2$ sufficiently small  so that  $\epsilon_1 + \epsilon_2 < \f{1}{1 + \epsilon}$, and using the discrete Gronwall inequality, we obtain the desired result of the theorem.
\end{proof}

As a direct consequence of  \Cref{main} and \Cref{discrete-split}, we have  the following first order error estimate for our numerical scheme \eqref{Numerical_Scheme}.
\vspace{0.5em}

\begin{theorem}\label{discrete}
Assume that  assumptions \textbf{(H1)-(H3)} hold. Then there exists a positive constant $C$ independent of $\Delta t$, such that
   \begin{equation}
   \max\limits_{1\leq i\leq N} E\big[(Y^i- Y_{t_i})^2 + \Delta t (Z^i -Z_{t_i})^2 \big] \leq C \Delta t^2.
   \end{equation}
\end{theorem}
\vspace{0.1em}

\section{Numerical experiments}
In this section, we use three numerical experiments to validate our splitting up scheme and verify the error estimations.
In order to implement the numerical schemes \eqref{Numerical_Scheme}, we need to approximate the conditional expectation $E_{t_i}$ in
 \cref{Numerical_Scheme}. Since the conditional expectation $E_{t_i}$ is essentially an integral with the Gaussian kernel,  we use Gauss - Hermite quadrature formula as a numerical integral method to calculate conditional expectations (see \cite{bao2016first} for more details). To calculate the general expectation $E$, we use Monte Carlo method with $300$ samples and compute  the root mean square error in each example.

 \subsection*{Example 1}

In the first example, we consider the BDSDE
\begin{equation*}
  \dd Y_t=-f(t,Y_t,Z_t)\dd t+Z_t\dd W_t-g(t,Y_t)\dd \overleftarrow{B}_t,
\end{equation*}
where $f(t, Y_t, Z_t) = \f{Y_t}{2} - Z_t + \f{B_t - B_T}{8}$ and $g(t, Y_t) = \f{1}{4} \big(\cos(t+W_t)^2 + Y_t - ( \f{B_t - B_T}{8})^2  \big)$.The exact solution to the above equation is $Y_t = \sin(t+W_t) - \f{B_t}{4} + \f{B_T}{4}$ and $Z_t = \cos(t+W_t)$.  
\begin{table}[ht]\small
\leftmargin=6pc \caption{Example 1} \label{e1:t1}
\begin{center}
\begin{tabular}{|c|c|c|c|c|c|c|}
 \hline  Partition &   $E[\|\tilde{Y}^0 - Y_{0}\|_{L^2}]$  & $E[\|Y^0 - Y_{0}\|_{L^2}]$  & $E[\|Z^0 - Z_{0}\|_{L^2}]$   \\
\hline   $N = 2^3$ & $2.4000e-01$   & $ 2.3489e-01$    & $  9.8106e-02$  \\
\hline   $N = 2^4$  & $1.3788e-01 $  & $  1.2745e-01  $     & $ 4.2258e-02 $  \\
\hline   $N = 2^5$ & $  7.7604e-02 $  & $ 6.4257e-02 $     & $ 2.0880e-02$  \\
\hline   $N = 2^6$  & $ 4.3455e-02  $  & $ 3.1834e-02 $    & $  1.2762e-02 $  \\
\hline   $N = 2^7$  & $  2.7816e-02 $   & $   1.5770e-02 $    & $  6.0392e-03$  \\
\hline   $CR$  & $ 0.79  $   & $  0.98  $    & $  0.98$  \\
\hline
\end{tabular}\end{center}
\end{table}
To demonstrate the performance of our numerical schemes, we compute the root mean square errors (RMSEs) between our approximate solutions and the exact solution. Specifically, we calculate the  expectation of the $L^2$ norm errors $\|\tilde{Y}^0 - Y_{0}\|_{L^2}$, $\|Y^0 - Y_{0}\|_{L^2}$ and $\|Z^0 - Z_{0}\|_{L^2}$ at time $t= 0$ with $300$ Monte Carlo samples and discretize the equations with time step sizes $\Delta t = 2^{-3}, 2^{-4}, 2^{-5}, 2^{-6}, 2^{-7}$. The corresponding errors are presented in Table \ref{e1:t1}. Here $CR$ in the table stands for ``convergence rate''. We can see from the table that $Y^i$ indeed provides a first order numerical approximation for the solution $Y_t$, and $\tilde{Y}^{i}$ provides reasonably accurate approximation for $Y_t$. However, since the scheme for $\tilde{Y}^{i}$ does not include  the $d \ola{B}$ integral, it does not provide a first order approximation for the solution. On the other hand, we can see that our numerical solution $Z^i$ converges with first order in approximating $Z_t$  in this example although in our proof we only obtain half order convergence analysis for $Z^i$. Further investigation is needed to determine if this a super convergence for $Z_t$ on the nodal points.

\subsection*{Example 2}

In the second example, we consider the BDSDE
\begin{equation*}
  \dd Y_t=-f(t,Y_t,Z_t)\dd t+Z_t\dd W_t-g(t,Y_t)\dd \overleftarrow{B}_t,
\end{equation*}
with $f(t, Y_t, Z_t) = (Y_t - t - B_t)^2 + \big(\cos(W_t) \big)^2 - \f{1}{2}\sin(W_t)$, $g(t, Y_t) = (Y_t - t - B_t)^2 + \big(\cos(W_t) \big)^2$. The exact solution is given by $Y_t = \sin(W_t) + t + B_t $ and $Z_t = \cos(W_t)$.
\begin{table}[ht]\small
\leftmargin=6pc \caption{Example 2} \label{e2:t1}
\begin{center}
\begin{tabular}{|c|c|c|c|c|c|c|}
\hline  Partition &   $E[\|\tilde{Y}^0 - Y_{0}\|_{L^2}]$  & $E[\|Y^0 - Y_{0}\|_{L^2}]$  & $E[\|Z^0 - Z_{0}\|_{L^2}]$   \\
\hline   $N =2^3$ &  $  4.1305e-01 $ & $ 2.4342e-01 $   & $  1.3810e-01 $  \\
\hline   $N =2^4$ & $   2.9031e-01 $ & $ 1.4541e-01 $    & $   9.0722e-02 $  \\
\hline   $N =2^5$ & $  1.9627e-01  $ & $  7.3160e-02 $   & $ 5.5327e-02 $  \\
\hline   $N =2^6$  & $ 1.3704e-01 $  & $  3.4548e-02 $   & $   2.7033e-02  $  \\
\hline   $N =2^7$  & $  8.9505e-02$  & $1.5549e-02  $   & $ 1.3429e-02 $  \\
\hline   $CR$  & $  0.55$  & $ 1.00$   & $ 0.85 $  \\
\hline
\end{tabular}\end{center}
\end{table}
We can see that in this example, both $f$ and $g$ are nonlinear function for $Y_t$. Therefore, this example  demonstrates the performance of our schemes in solving nonlinear BDSDE systems.
As in the first example, we evaluate he RMSEs $E[\|\tilde{Y}^0 - Y_{0}\|_{L^2}]$, $E[\|Y^0 - Y_{0}\|_{L^2}]$ and $E[\|Z^0 - Z_{0}\|_{L^2}]$, at time $t= 0$. In Table \ref{e2:t1}, we can see that the convergence order for $\|Y^0 - Y_{0}\|_{L^2}$ is $1$ and the convergence order for $\|\tilde{Y}^0 - Y_{0}\|_{L^2}$ is roughly $0.55$,  For the numerical solution $Z^i$, we can see from the table that the convergence for $\|Z^0 - Z_{0}\|_{L^2}$ is $0.847$, which is less than $1$. From this example we can see that $Z^i$ does not always produce first order approximation for $Z_{t_i}$.

\subsection*{Example 3}


In the third example, we consider the BDSDE
\begin{equation*}
  \dd Y_t=-f(t,Y_t,Z_t)\dd t+Z_t\dd W_t-g(t,Y_t)\dd \overleftarrow{B}_t,
\end{equation*}
with $f(t, Y_t, Z_t) = - \f{1}{2} ( \sin(Y_t) )^2 - \f{1}{2}(\cos(t + W_t + \f{1}{2} B_t))^2 - \f{1}{2} (Z_t)^2$ and $g(t, Y_t) = - \f{1}{2} ( \sin(Y_t) )^2 - \f{1}{2}(\cos(t + W_t + \f{1}{2} B_t))^2$. The exact solution for the above equation is $Y_t = t + W_t + \f{1}{2} B_t$ and $Z_t = 1$.
\begin{table}[ht]\small
\leftmargin=6pc \caption{Example 3} \label{e3:t1}
\begin{center}
\begin{tabular}{|c|c|c|c|c|c|c|}
\hline  Partition &   $E[\|\tilde{Y}^0 - Y_{0}\|_{L^2}]$  & $E[\|Y^0 - Y_{0}\|_{L^2}]$  & $E[\|Z^0 - Z_{0}\|_{L^2}]$   \\
\hline   $N = 2^3$ &  $ 1.7759e-01 $ & $  9.2446e-03  $   & $  1.6664e-02 $  \\
\hline   $N = 2^4$ & $ 1.1907e-01  $ & $  4.3131e-03 $    & $ 8.1184e-03   $  \\
\hline   $N = 2^5$ & $ 9.0338e-02 $ & $  2.2815e-03  $   & $ 4.2914e-03$  \\
\hline   $N = 2^6$  & $  5.7300e-02 $  & $ 1.1198e-03 $   & $  2.1936e-03 $  \\
\hline   $N = 2^7$  & $ 4.6539e-02$  & $ 5.5550e-04$   & $ 1.0452e-03$  \\
\hline   $CR$  & $ 0.49 $  & $ 1.01  $   & $ 0.99 $  \\
\hline
\end{tabular}\end{center}
\end{table}
In the last  example, $f$ is a nonlinear functions for $Y_t$, $Z_t$, $g$ is a nonlinear function for $Y_t$, and $Y_t$ is in a trigonometric  function in both $f$ and $g$.  The purpose of this example to demonstrate the performance of our method in solving a  more general BDSDE system.  In Table \ref{e3:t1}, we present the RMSEs between our approximate solutions and the exact solution at time $t= 0$.  From this table, we can see that for this example $\tilde{Y}^0$ converges to $Y_0$ with half order, and $( \hat{Y}_0, \hat{Z}_0)$ converges to $(Y_0, Z_0)$with first order.

\nocite{*}
\bibliographystyle{plain}

\begin{thebibliography}{10}

\bibitem{Bachouch2016}
Achref Bachouch, Mohamed~Anis Ben~Lasmar, Anis Matoussi, and Mohamed Mnif.
\newblock Euler time discretization of backward doubly {SDE}s and application
  to semilinear {SPDE}s.
\newblock {\em Stochastics and Partial Differential Equations: Analysis and
  Computations}, 4(3):592--634, Sep 2016.
  
  
\bibitem{bao2020_Jump}
Feng Bao, Richard Archibald, and Peter Maksymovych
\newblock Backward SDE Filter for Jump Diffusion Processes and Its Applications in Material Sciences
\newblock {\em Communications in Computational Physics}, 27:589-618,
  2020.
  
\bibitem{bao2019_Jump}
Feng Bao, Yanzhao Cao and Hongmei Chi.
\newblock Adjoint Forward Backward Stochastic Differential Equations Driven
by Jump Processes and Its Application to Nonlinear Filtering Problems
\newblock {\em International Journal for Uncertainty Quantification}, 9(2):143-159, 2019.



\bibitem{bao2019}
Feng Bao, Yanzhao Cao and Xiaoying Han.
\newblock Forward backward doubly stochastic differential equations and the
  optimal filtering of diffusion processes.
\newblock {\em Communications in Mathematical Sciences}, 18(3):635-661, 2020.


\bibitem{bao2014_JUQ}
Feng Bao, Yanzhao Cao, Clayton Webster, and Guannan Zhang.
\newblock A Hybrid Sparse-Grid Approach for Nonlinear Filtering Problems Based on Adaptive-Domain of the Zakai Equation Approximations,
\newblock {\em SIAM/ASA Journal on Uncertainty Quantification}, 2(1):784-804,
  2014.
  
  
\bibitem{bao2011_BDSDE}
Feng Bao, Yanzhao Cao, and Weidong Zhao.
\newblock Numerical Solutions for Forward Backward Doubly Stochastic
Differential Equations and Zakai Equations
\newblock {\em International Journal for Uncertainty Quantification}, 4(1):351-367,
  2011.
  
  
 \bibitem{bao2015_First}
Feng Bao, Yanzhao Cao, and Weidong Zhao.
\newblock A Backward Doubly Stochastic Differential Equation Approach for Nonlinear Filtering Problems
\newblock {\em Communications in Computational Physics}, 23(5):1573-1601,
  2018.
  
  \bibitem{bao2018_First}
Feng Bao, Yanzhao Cao, and Weidong Zhao.
\newblock A First Order Semi-discrete Algorithm for Backward Doubly Stochastic Differential Equations,
\newblock {\em Discrete and Continuous Dynamical Systems - Series B}, 5(2):1297-1313,
  2015.


\bibitem{bao2016first}
Feng Bao, Yanzhao Cao, Amnon Meir, and Weidong Zhao.
\newblock A first order scheme for backward doubly stochastic differential
  equations.
\newblock {\em SIAM/ASA Journal on Uncertainty Quantification}, 4(1):413-445,
  2016.
  
  \bibitem{Bao_2017}
Feng Bao and Vasileios Maroulas.
\newblock Adaptive meshfree backward sde filter.
\newblock {\em SIAM Journal on Scientific Computing}, 39(6):A2664--A2683, 2017.



\bibitem{bensoussan1989approximation}
Allan Bensoussan and R~Glowinski.
\newblock Approximation of {Z}akai equation by the splitting up method.
\newblock In {\em Stochastic Systems and Optimization}, pages 255--265.
  Springer, 1989.

\bibitem{bensoussan1990approximation}
Allan Bensoussan, R~Glowinski, and A~Rascanu.
\newblock Approximation of the {Z}akai equation by the splitting up method.
\newblock {\em SIAM Journal on Control and Optimization}, 28(6):1420--1431,
  1990.

\bibitem{bensoussan1992approximation}
Allan Bensoussan, R~Glowinski, and A~Ra{\c{s}}canu.
\newblock Approximation of some stochastic differential equations by the
  splitting up method.
\newblock {\em Applied Mathematics and Optimization}, 25(1):81--106, 1992.

\bibitem{bouchard2004discrete}
Bruno Bouchard and Nizar Touzi.
\newblock Discrete-time approximation and {M}onte-{C}arlo simulation of
  backward stochastic differential equations.
\newblock {\em Stochastic Processes and their applications}, 111(2):175--206,
  2004.

\bibitem{douglas1996numerical}
Jim Douglas, Jin Ma, Philip Protter, et~al.
\newblock Numerical methods for forward-backward stochastic differential
  equations.
\newblock {\em The Annals of Applied Probability}, 6(3):940--968, 1996.

\bibitem{E2017}
Weinan E, Jiequn Han, and Arnulf Jentzen.
\newblock Deep learning-based numerical methods for high-dimensional parabolic
  partial differential equations and backward stochastic differential
  equations.
\newblock {\em Commun. Math. Stat.}, 5(4):349--380, 2017.

\bibitem{germani1988semi}
Aurora Germani and M~Piccioni.
\newblock Semi-discretization of stochastic partial differential equations on
  {R}$^d$ by a finite-element technique a. germani.
\newblock {\em Stochastics: An International Journal of Probability and
  Stochastic Processes}, 23(2):131--148, 1988.

\bibitem{doi:10.1137/16M106371X}
E.~Gobet, J.~López-Salas, P.~Turkedjiev, and C.~Vázquez.
\newblock Stratified regression {M}onte-{C}arlo scheme for semilinear {PDE}s
  and {BSDE}s with large scale parallelization on {GPU}s.
\newblock {\em SIAM Journal on Scientific Computing}, 38(6):C652--C677, 2016.

\bibitem{GOBET20171171}
E.~Gobet and P.~Turkedjiev.
\newblock Adaptive importance sampling in least-squares {M}onte {C}arlo
  algorithms for backward stochastic differential equations.
\newblock {\em Stochastic Processes and their Applications}, 127(4):1171 --
  1203, 2017.

\bibitem{gyongy2003splitting}
Istv{\'a}n Gy{\"o}ngy and Nicolai Krylov.
\newblock On the splitting-up method and stochastic partial differential
  equations.
\newblock {\em The Annals of Probability}, 31(2):564--591, 2003.

\bibitem{Han2018}
Jiequn Han, Arnulf Jentzen, and Weinan E.
\newblock Solving high-dimensional partial differential equations using deep
  learning.
\newblock {\em Proc. Natl. Acad. Sci. USA}, 115(34):8505--8510, 2018.

\bibitem{jentzen2010}
Arnulf Jentzen and Peter Kloeden.
\newblock Taylor expansions of solutions of stochastic partial differential
  equations with additive noise.
\newblock {\em Ann. Probab.}, 38(2):532--569, 03 2010.

\bibitem{kloeden2013numerical}
Peter Kloeden and Eckhard Platen.
\newblock {\em Numerical solution of stochastic differential equations},
  volume~23.
\newblock Springer Science \& Business Media, 2013.

\bibitem{labart2013parallel}
C{\'e}line Labart and J{\'e}r{\^o}me Lelong.
\newblock A parallel algorithm for solving {BSDE}s.
\newblock {\em Monte Carlo Methods and Applications}, 19(1):11--39, 2013.

\bibitem{legland1992splitting}
Fran{\c{c}}ois LeGland.
\newblock Splitting-up approximation for {SPDE}'s and {SDE}'s with application
  to nonlinear filtering.
\newblock In {\em Stochastic partial differential equations and their
  applications}, pages 177--187. Springer, 1992.

\bibitem{lions1998fully}
Pierre-Louis Lions and Panagiotis~E Souganidis.
\newblock Fully nonlinear stochastic partial differential equations.
\newblock {\em Comptes Rendus de l'Acad{\'e}mie des Sciences-Series
  I-Mathematics}, 326(9):1085--1092, 1998.

\bibitem{doi:10.1137/S0363012993248918}
S.~Lototsky, R.~Mikulevicius, and B.~Rozovskii.
\newblock Nonlinear {F}iltering {R}evisited: {A} {S}pectral {A}pproach.
\newblock {\em SIAM Journal on Control and Optimization}, 35(2):435--461, 1997.

\bibitem{ma2002numberical}
Jin Ma, Philip Protter, Jaime San~Martin, Soledad Torres, et~al.
\newblock Numberical method for backward stochastic differential equations.
\newblock {\em The Annals of Applied Probability}, 12(1):302--316, 2002.

\bibitem{ma1994solving}
Jin Ma, Philip Protter, and Jiongmin Yong.
\newblock Solving forward-backward stochastic differential equations
  explicitly—a four step scheme.
\newblock {\em Probability theory and related fields}, 98(3):339--359, 1994.

\bibitem{milstein2009solving}
G~Milstein and M~Tretyakov.
\newblock Solving parabolic stochastic partial differential equations via
  averaging over characteristics.
\newblock {\em Mathematics of computation}, 78(268):2075--2106, 2009.

\bibitem{nagase1995remarks}
Noriaki Nagase.
\newblock Remarks on nonlinear stochastic partial differential equations: an
  application of the splitting-up method.
\newblock {\em SIAM journal on control and optimization}, 33(6):1716--1730,
  1995.

\bibitem{Pardoux1987}
E.~Pardoux and P.~Protter.
\newblock A two-sided stochastic integral and its calculus.
\newblock {\em Probability Theory and Related Fields}, 76(1):15--49, Sep 1987.

\bibitem{pardoux1992backward}
Etienne Pardoux and Shige Peng.
\newblock Backward stochastic differential equations and quasilinear parabolic
  partial differential equations.
\newblock In {\em Stochastic partial differential equations and their
  applications}, pages 200--217. Springer, 1992.

\bibitem{pardoux1994backward}
{E}tienne Pardoux and Shige Peng.
\newblock Backward doubly stochastic differential equations and systems of
  quasilinear {SPDE}s.
\newblock {\em Probability Theory and Related Fields}, 98(2):209--227, 1994.

\bibitem{pardouxt1980stochastic}
E~Pardouxt.
\newblock Stochastic partial differential equations and filtering of diffusion
  processes.
\newblock {\em Stochastics}, 3(1-4):127--167, 1980.

\bibitem{yoo2000semi}
Hyek Yoo.
\newblock Semi-discretization of stochastic partial differential equations on
  {R}$^1$ by a finite-difference method.
\newblock {\em Mathematics of computation}, 69(230):653--666, 2000.

\bibitem{zakai1969optimal}
Moshe Zakai.
\newblock On the optimal filtering of diffusion processes.
\newblock {\em Zeitschrift f{\"u}r Wahrscheinlichkeitstheorie und verwandte
  Gebiete}, 11(3):230--243, 1969.

\bibitem{JCM-31-221}
Guannan Zhang, Max Gunzburger, and Weidong Zhao.
\newblock A sparse-grid method for multi-dimensional backward stochastic
  differential equations.
\newblock {\em Journal of Computational Mathematics}, 31(3):221--248, 2013.

\bibitem{zhang2004}
Jianfeng Zhang.
\newblock A numerical scheme for {BSDE}s.
\newblock {\em Ann. Appl. Probab.}, 14(1):459--488, 02 2004.

\bibitem{zhang2017backward}
Jianfeng Zhang.
\newblock Backward stochastic differential equations.
\newblock In {\em Backward Stochastic Differential Equations}, pages 79--99.
  Springer, 2017.

\end{thebibliography}
\end{document}